\documentclass{amsart}

\theoremstyle{definition}
\newtheorem{definition}{Definition}[section]
\theoremstyle{theorem}
\newtheorem{theorem}{Theorem}[section]
\newtheorem{proposition}[theorem]{Proposition}
\newtheorem{corollary}[theorem]{Corollary}
\newtheorem{lemma}[theorem]{Lemma}
\newtheorem{remark}[theorem]{Remark}
\newcommand{\fra}[2]{\displaystyle{\frac{#1}{#2}}}
\newcommand{\ka}{\kappa}

\begin{document}
\title{The curvature tensor of $(\ka,\mu,\nu)$-contact metric manifolds}



\begin{abstract} We study the Riemann curvature tensor of \emph{$(\ka,\mu,\nu)$-contact metric manifolds}, which we prove to be completely determined in dimension $3$, and we observe how it is affected by
$D_a$-homothetic deformations. This prompts the definition and study of
\emph{generalized $(\ka,\mu,\nu)$-space forms} and of the necessary and sufficient conditions for them to be conformally flat.
\end{abstract}

\author[K. Arslan]{Kadri Arslan}
\address{Department of Mathematics\\ Faculty of Art and Sciences\\ Uludag University\\ Bursa, TURKEY.}
\email{arslan@uludag.edu.tr}

\author[A. Carriazo]{Alfonso Carriazo}
\address{Department of Geometry and Topology\\ Faculty of Mathematics\\ University of Sevilla\\ Apdo. de Correos 1160, 41080 -- Sevilla, SPAIN.}
\email{carriazo@us.es}

\author[V. Mart\'in-Molina]{Ver\'onica Mart\'in-Molina}
\address{Department of Geometry and Topology\\ Faculty of Mathematics\\ University of Sevilla\\ Apdo. de Correos 1160, 41080 -- Sevilla, SPAIN.}
\email{veronicamartin@us.es}

\author[C. Murathan]{Cengizhan Murathan}
\address{Department of Mathematics\\ Faculty of Art and Sciences\\ Uludag University\\ Bursa, TURKEY.}
\email{cengiz@uludag.edu.tr}

\maketitle

\section{Introduction}

All researchers who are currently working on contact metric
geometry and related topics agree on the great importance of
\emph{$(\kappa,\mu)$-spaces}, since they were introduced by D. E.
Blair, T. Koufogiorgos and V. J. Papantoniou in \cite{bisr} as
those contact metric manifolds satisfying the equation
\begin{equation}\label{kappamu}
R(X,Y)\xi = \ka \{ \eta (Y)X-\eta (X) Y\} + \mu\{\eta (Y)hX-\eta
(X)hY\},
\end{equation}
for every  $X,Y$ on $M$, where $\ka$ and $\mu$ are constants,
$h=1/2 L_\xi \phi$ and $L$ is the usual Lie derivative. These
spaces include the Sasakian manifolds ($\kappa=1$ and $h=0$), but
the non-Sasakian examples have proven to be even more interesting.
Actually, their name was really given by E. Boeckx in \cite{boe2},
who also provided a classification the next year in \cite{boe}. R.
Sharma extended the notion in \cite{sharma}, by considering $\ka$
and $\mu$ to be differentiable functions on the manifold and
called those new spaces \emph{generalized $(\ka,\mu)$-spaces}.
Later, T. Koufogiorgos and C. Tsichlias proved in \cite{kouf2000}
that in dimensions greater than or equal to $5$, the functions
$\ka,\mu$ must be constant and presented examples in dimension $3$
with non-constant functions. There have been more papers dealing
with these spaces, some of them replacing the contact metric
structure by a different one, but let us emphasize the recent work
published by B. Cappelletti Montano and L. Di Terlizzi in
\cite{mino} as a proof of their relevance and possibilities.

Starting from the paper \cite{kouf97}, in which T. Koufogiougos
gave an expression for the curvature tensor of a
\emph{$(\ka,\mu)$-space} with pointwise constant $\phi$-sectional
curvature and dimension greater than or equal to $5$, the second
and third authors (jointly with M. M. Tripathi) recently defined
in \cite{C+MM} a {\em generalized} $\left( \kappa ,\mu \right)
${\em -space form} as an almost contact metric manifold
$(M^{2n+1},\phi ,\xi ,\eta ,g)$ whose curvature tensor can be
written as
\begin{equation} \label{def-space}
R=f_{1}R_{1}+f_{2}R_{2}+f_{3}R_{3}+f_{4}R_{4}+f_{5}R_{5}+f_{6}R_{6},
\end{equation}
where $f_{1},\ldots,f_{6}$ are differentiable functions on $M$ and
$R_{1},\ldots,R_{6}$ are the  tensors given by
\begin{align*}
R_{1}( X,Y) Z&=g( Y,Z ) X-g ( X,Z )Y,\\
R_{2} ( X,Y ) Z&=g ( X,\phi Z ) \phi Y-g (Y,\phi Z ) \phi X+2g ( X,\phi Y ) \phi Z,\\
R_{3} ( X,Y ) Z&=\eta (X)\eta  ( Z ) Y-\eta ( Y ) \eta  ( Z ) X +g ( X,Z ) \eta ( Y ) \xi -g ( Y,Z ) \eta  ( X )\xi,\\
R_{4} ( X,Y ) Z &=g ( hY,Z ) hX-g (hX,Z ) hY +g ( \phi hX,Z ) \phi hY-g ( \phi hY,Z ) \phi hX, \\
R_{5} ( X,Y ) Z &=g ( Y,Z ) hX -g (X,Z ) hY+g ( hY,Z ) X-g ( hX,Z ) Y, \\
R_{6} ( X,Y ) Z &=\eta  ( X ) \eta  (Z ) hY-\eta  ( Y ) \eta  ( Z
) hX +g (hX,Z ) \eta  ( Y ) \xi -g ( hY,Z ) \eta ( X ) \xi,
\end{align*}%
for any vector fields $X,Y,Z$. Such a manifold was denoted by
$M(f_1, \ldots ,f_6)$ and several examples of it were presented in
\cite{C+MM}. This notion also includes that of \emph{generalized
Sasakian-space-forms}, which can be obtained by putting
$f_4=f_5=f_6=0$ in (\ref{def-space}). For more details about these
spaces, see \cite{b1} and \cite{b3}.

Also in \cite{C+MM}, after the formal definition of a
\emph{generalized $( \ka ,\mu ) $-space form} was given, those
with contact metric structure  were deeply studied. It was proved
that they are \emph{generalized $( \kappa ,\mu )$-spaces} with
$\kappa =f_{1}-f_{3}$ and $\mu =f_{4}-f_{6}$. Furthermore, if
their dimension is greater than or equal to $5$, then they are
$(-f_{6},1-f_{6})$-spaces with constant $\phi $-sectional
curvature $2f_{6}-1$, where $f_{4}=1$, $f_{5}=1/2$ and
$f_{1},f_{2},f_{3}$ depend linearly on the constant $f_{6}$. A
method for constructing infinitely many examples of this type was
also presented.

Moreover, it was proved that the curvature tensor of a
\emph{generalized $(\ka,\mu)$-space form} is not unique in the
$3$-dimensional case and that several properties and results are
also satisfied. Examples of \emph{generalized $( \kappa ,\mu )
$-space forms} with non-constant functions $f_{1},f_{3}$ and
$f_{4}$ were also given.

Later, in \cite{C+MM2} the study of \emph{generalized
$(\ka,\mu)$-space forms} was continued by analysing the behaviour
of such spaces under $D_a$-homothetic deformations. An alternative
definition of this type of manifold was introduced and it was
proved that these spaces remain so after a $D_a$-homothetic
deformation, albeit with different functions. Infinitely many
examples of this type of manifold were also showed in dimension
$3$ with some non-constant functions.

Going a step further from \emph{$(\ka,\mu)$-spaces}, T.
Koufogiorgos, M. Markellos and V. J. Papantoniou introduced in
\cite{kouf2008} the notion of \emph{$(\ka,\mu,\nu)$-contact metric
manifold}, where now the equation to be satisfied is
\begin{equation} \label{kappamunu}
\begin{split}
R( X,Y) \xi &=\kappa \{ \eta ( Y) X-\eta ( X) Y\} +\mu \{ \eta ( Y) hX-\eta ( X) hY\} \\
&+\nu \{ \eta ( Y) \phi hX-\eta ( X) \phi hY\},
\end{split}
\end{equation}
for some smooth functions $\kappa, \mu $, and $\nu $ on $M$. They
proved that, in dimension greater than or equal to $5$, $\ka$ and
$\mu$ are necessarily constant and that $\nu$ is zero, hence the
\emph{$(\ka,\mu,\nu)$-contact metric manifolds} are in particular
\emph{$(\ka,\mu)$-spaces}. They also proved that if a
$D_a$-homothetic deformation is applied to them, they keep being
\emph{$(\ka,\mu,\nu)$-contact metric manifolds}, although with
different functions, result that they used to provide examples in
dimension $3$ with $\nu$ a non-zero function. Some other authors
also studied manifolds satisfying condition (\ref{kappamunu}), but
with a non-contact metric structure, as we will point out later.

In the present paper, after reviewing some concepts and results on
almost contact metric manifolds in section
\ref{section-preliminaries}, we prove in section
\ref{section-curvature} that the curvature tensor of a
3-dimensional  \emph{$(\ka,\mu,\nu)$-contact metric manifold} is
completely determined and can be written in terms of $\ka,\mu,\nu$
and its $\phi$-sectional curvature $F$. We apply this result in
order to  give the curvature tensor in a particular example and we
study how $D_a$-homothetic deformations affect it. In section
\ref{section-generalized} we define \emph{generalized
$(\ka,\mu,\nu)$-space forms} as a generalization of
\emph{generalized $(\ka,\mu)$-space forms} that englobes the
$(\ka,\mu,\nu)$-\emph{contact metric manifolds} and we provide
some properties and examples. Finally, in section
\ref{section-conformally} we study some necessary and sufficient
conditions for \emph{generalized $(\ka,\mu,\nu)$-space forms} of
dimension greater or equal to $5$ to be conformally flat.

In conclusion, by introducing \emph{generalized
$(\ka,\mu,\nu)$-space forms} we offer a very general frame in
which many previous theories can be included and unified, opening
new possibilities for further studies.

\section{Preliminaries}\label{section-preliminaries}

In this section, we recall some general definitions and basic
formulas which will be used later. For more background on almost
contact metric manifolds, we recommend the reference
\cite{blairb}.

An odd-dimensional Riemann manifold $(M,g)$ is said to be an
{\em almost
contact metric manifold} if there exist on $M$ a $(1,1)$-tensor field $%
\phi $, a vector field $\xi $ (called the {\em structure vector
field}) and a 1-form $\eta $ such that $\eta (\xi )=1$, $\phi
^{2}X=-X+\eta ( X) \xi $ and $g(\phi X,\phi Y)=g(
X,Y) -\eta ( X) \eta ( Y) $ for any
vector fields $X,Y$ on $M$. In particular, in an almost contact
metric manifold we also have $\phi \xi =0$ and $\eta \circ \phi
=0$.

Such a manifold is said to be a {\em contact metric manifold} if ${\rm d}%
\eta =\Phi $, where $\Phi ( X,Y) =g(X,\phi Y)$ is the {\em %
fundamental} $2${\em -form} of $M$.
If, in addition, $\xi $ is a
Killing vector field, then $M$ is said to be a $K$-{\em contact
manifold}. It is well-known that a contact metric manifold is a
$K$-contact manifold if and only if
\begin{equation}
\nabla _{X}\xi =-\,\phi X  \label{eq-K-cont}
\end{equation}
for all vector fields $X$ on $M$. Even an almost contact metric
manifold satisfying the equation (\ref{eq-K-cont}) becomes a
$K$-contact manifold.

On the other hand, the almost contact metric structure of $M$ is
said to be {\em normal} if the Nijenhuis torsion $[\phi ,\phi ]$\
of $\phi $ equals $-2{\rm d}\eta \otimes \xi $. A normal contact
metric manifold is called a {\em Sasakian manifold}. It can be
proved that an almost contact metric manifold is Sasakian if and
only if
\begin{equation}
(\nabla _{X}\phi )Y=g( X,Y) \xi -\eta ( Y) X
\label{eq-Sas}
\end{equation}%
for any vector fields $X,Y$ on $M$. Moreover, for a Sasakian
manifold the following equation holds: $$ R( X,Y) \xi
=\eta ( Y) X-\eta ( X) Y.$$

Given an almost contact metric manifold $(M,\phi ,\xi ,\eta ,g)$, a $%
\phi $-{\em section} of $M$ at $p\in M$ is a section $\Pi
\subseteq
T_{p}M $ spanned by a unit vector $X_{p}$ orthogonal to $\xi _{p}$, and $%
\phi X_{p}$. The $\phi $-{\em sectional curvature of} $\Pi $ is
defined by $K(X,\phi X)=R(X,\phi X,\phi X,X)$. A Sasakian manifold
with constant $\phi $-sectional curvature $c$ is called a {\em
Sasakian space form}. In such a case, its Riemann curvature
tensor is given by equation $R=f_1R_1+f_2R_2+f_3R_3$ with
functions $f_{1} = (c+3)/4$, $f_{2} = f_{3}= (c-1)/4$.

It is well known that on a contact metric manifold $( M,\phi
,\xi ,\eta ,g) $, the tensor $h$, defined by $2h=L_\xi
\phi$, satisfies the following relations \cite{bisr}
\begin{equation}
h\xi =0,\quad \nabla _{X}\xi =-\phi X-\phi hX,\quad h\phi =-\phi
h,\quad {\rm tr}h=0,\quad \eta \circ h=0. \label{eq-cont-h}
\end{equation}%
Therefore, it follows from equations (\ref{eq-K-cont}) and
(\ref{eq-cont-h}) that a contact metric manifold is $K$-contact if
and only if $h=0$.

On the other hand, a contact metric manifold $(M^{2n+1},\phi ,\xi
,\eta ,g)$ is said to be a {\em generalized $( \kappa ,\mu )
$-space} if its  curvature tensor satisfies the condition
(\ref{kappamu}) for some smooth functions $\kappa $ and $\mu $ on
$M$ independent of the choice of vectors fields $X$ and $Y$. If
$\kappa $ and $\mu $ are constant, the manifold is called a $(
\kappa ,\mu ) $-{\em space}. T. Koufogiorgos proved in
\cite{kouf97} that if a $( \kappa ,\mu ) $-space $M$ has pointwise
constant $\phi $-sectional curvature $F$ and dimension greater
than or equal to $5$, the curvature tensor of this $( \kappa ,\mu
) $-space form is given by equation (\ref{def-space}), where
\begin{equation}\label{eq-kouf}
f_{1}=\frac{F+3}{4},\ f_{2}=\frac{F-1}{4},\ f_{3}=\frac{F+3}{4}-\kappa
,\ f_{4}=\frac{1}{2},\ f_{5}=1,\ f_{6}=1-\mu.
\end{equation}
Recently, a  $( \kappa ,\mu, \nu )$-{\em contact metric manifold}
was defined in \cite{kouf2008} as a contact metric manifold
$(M,\phi,\xi,\eta,g)$ whose curvature tensor satisfies
(\ref{kappamunu}) for some smooth functions $\kappa, \mu $, and
$\nu $ on $M$ independent of the choice of vectors fields $X$ and
$Y$. According to the above notations, we could also refer to them
as contact metric {\em generalized $(\kappa,\mu,\nu)$-spaces}.

It was showed in \cite{kouf2008} that every
$(\ka,\mu,\nu)$-contact metric manifold of dimension greater than or equal to $5$ is a $(\ka,\mu)$-space, but that there exist examples in
dimension $3$ with $\nu \neq 0$.

\

Given an almost contact metric manifold $(M,\phi,\xi,\eta,g)$, we
recall that a $D_a$-homo\-thetic deformation is defined by
\begin{equation} \label{eq-deformation} \overline{\phi}=\phi, \quad
\overline{\xi}=\frac{1}{a}\xi, \quad \overline{\eta}=a\eta, \quad
\overline{g}= ag + a(a-1)\eta \otimes \eta,
\end{equation}
where $a$ is a positive constant (see  \cite{tanno2}). It is clear
that
$(M,\overline{\phi},\overline{\xi},\overline{\eta},\overline{g})$
is also an almost contact metric manifold and that
\begin{equation} \label{hbarra} \overline{h}=\fra{1}{a} h.
\end{equation}

\

Finally, we will denote by $Q$ the Ricci operator on $M$ and define the scalar curvature as $\tau=trQ$. We will also assume that all the functions considered in this paper
will be differentiable functions on the corresponding manifolds.

\section{The curvature tensor of $(\ka,\mu, \nu)$-contact metric manifolds}\label{section-curvature}

In this section we will study the curvature tensor of
$(\ka,\mu,\nu)$-\emph{contact metric manifolds}, which is completely
determined in dimension $3$. We will also see how a $D_a$-homothetic
deformation  affects it.

We know from \cite{kouf2008} that a $(\ka,\mu,\nu)$-\emph{contact metric manifold} satisfies $\ka \leq 1$ and that the condition $\ka=1$ is equivalent to being Sasakian. Therefore, we will concentrate on the case $\ka<1$.

\begin{theorem}\label{prop-R}
Let $M$ be a 3-dimensional $(\ka,\mu,\nu)$-\emph{contact metric manifold}
with $\ka<1$. Then its curvature tensor can be written as
$$R=\left( \fra{\tau}{2} -2 \ka \right) R_1+\left( \fra{\tau}{2} -3 \ka \right)R_3+\mu R_4 +\nu R_7,$$ where $R_1,R_3,R_4$ are the same
tensors appearing in (\ref{def-space})  and $R_7$ is the
following one:
\begin{equation} \label{def-R7}
R_7(X,Y)Z=g(Y,Z)\phi h X-g(X,Z)\phi h Y+g(\phi h Y,Z)X-g(\phi h
X,Z)Y.
\end{equation}
\end{theorem}
\begin{proof}
It is well known that a 3-dimensional contact metric manifold
satisfies:
\begin{equation}\label{eq-dim3-contact}
\begin{aligned}
R(X,Y)Z&=g(Y,Z)QX-g(X,Z)QY+g(QY,Z)X-g(QX,Z)Y\\
&-\fra{\tau}{2}(g(Y,Z)X-g(X,Z)Y).
\end{aligned}
\end{equation}

Thanks to Proposition 3.1 from \cite{kouf2008} we also know that:
\begin{equation}\label{eq-Q-dim3}
Q=\left( \fra{\tau}{2} - \ka \right) I+\left( -\fra{\tau}{2} +3 \ka \right) \eta \otimes \xi+\mu h+\nu \phi h.
\end{equation}

Substituting equation (\ref{eq-Q-dim3}) in (\ref{eq-dim3-contact})
we obtain:
\begin{align*}
R(X,Y)Z&=\left( \fra{\tau}{2} -2 \ka \right) R_1 (X,Y)Z+\left(\fra{\tau}{2} -3 \ka \right)R_3 (X,Y)Z+\mu R_4 (X,Y)Z \\
&+\nu \ \{ g(Y,Z)\phi h X-g(X,Z)\phi h Y+g(\phi h Y,Z)X-g(\phi h
X,Z)Y \}.
\end{align*}

We only need to define the tensor $R_7$ as written in \eqref{def-R7} in order to get
the desired result.
\end{proof}

\

We can also prove the following:

\begin{proposition}\label{prop-F}
Let $M$ be a $3$-dimensional $(\ka,\mu,\nu)$-\emph{contact metric manifold} with $\ka <
1$. Then its $\phi$-sectional curvature is $F= \fra{\tau}{2} -2 \ka
$.
\end{proposition}
\begin{proof}
There exists a $\phi$-basis $\{e,\phi e, \xi \}$ with $h X=\lambda
X$ (where $\lambda=\sqrt{1-\ka}$) because $\ka < 1$ (equation
(4-8) from \cite{kouf2008}). Due to the fact that the
$\phi$-sectional curvature $F=R(X,\phi X,\phi X,X)$  on a point $P
\in M$ does not depend on the choice of $X$, then $F=R(e,\phi
e,\phi e,e)$.

If we use now Proposition \ref{prop-R} we get:

\begin{align*}
F=R(e,\phi e,\phi e ,e)=&\left( \fra{\tau}{2} -2 \ka \right) R_1
(e,\phi e,\phi e,e)+\left( \fra{\tau}{2} -3 \ka \right) R_3 (e,\phi
e,\phi e,e)\\
&+\mu R_4 (e,\phi e,\phi e,e)+\nu R_7 (e,\phi e,\phi e,e).
\end{align*}

An straightforward computation gives us that
\begin{align*}
&R_1 (e,\phi e,\phi e,e)=1,\\
&R_3 (e,\phi e,\phi e,e)=R_4(e,\phi e,\phi e,e)=R_7(e,\phi e,\phi e,e)=0.
\end{align*}

We conclude that $$F=R(e,\phi e,\phi e ,e)= \fra{\tau}{2} -2 \ka,$$ as stated above.
\end{proof}

Therefore, Proposition \ref{prop-R} can be rewritten as:

\begin{corollary}\label{prop-R-v2}
Let $M$ be a 3-dimensional $(\ka,\mu,\nu)$-\emph{contact metric manifold}
with $\ka<1$. Then its curvature tensor can be written as
\begin{equation}\label{eq-R-v2}
R=F R_1+(F-\ka) R_3+\mu R_4 +\nu R_7,
\end{equation} where $F$ is
the $\phi$-sectional curvature and $R_1,R_3,R_4,R_7$ are the
previously defined tensors.
\end{corollary}

\begin{remark}
Corollary \ref{prop-R-v2} could also be obtained analogously to
how it was proved in \cite{kouf97} that a $(\ka,\mu)$-\emph{space form} of dimension greater than or equal to $5$ has curvature tensor $R=f_1 R_1+\cdots+f_6 R_6$, with $f_1,\ldots,f_6$ functions as in \eqref{eq-kouf}.

The  hypothesis on the dimension was only used to prove that the $\phi$-sectional curvature is constant, not the form of the tensor $R$, so the reasoning is also valid in dimension $3$, where $K(X,\phi X)$ is always independent of the choice of $X$. Adapting that proof to the case of the  $(\ka,\mu,\nu)$-\emph{contact metric manifolds}, we would obtain that the formula of $R(\widetilde{X},\widetilde{Y})Z$ does not vary if  $\widetilde{X},\widetilde{Y}$ are vector fields orthogonal  to $\xi$:
\begin{equation*}
R(\widetilde{X},\widetilde{Y})Z=
\fra{F+3}{4} R_1(\widetilde{X},\widetilde{Y})Z
+\fra{F-1}{4} R_2(\widetilde{X},\widetilde{Y})Z
+R_4(\widetilde{X},\widetilde{Y})Z
+\fra{1}{2} R_5(\widetilde{X},\widetilde{Y})Z.
\end{equation*}

If we use the fact that every vector field can be written as $X=\widetilde{X}+\eta(X) \xi$, where $\widetilde{X}$ is orthogonal to $\xi$, and the formula of $R(\xi,X)Y$ for a $(\ka,\mu,\nu)$-\emph{contact metric manifold} (see equation (4-10) of \cite{kouf2008}), we get
\begin{align*}
R(X,Y)Z=&R(\widetilde{X},\widetilde{Y})Z-\eta(Y)R(\xi,\widetilde{X})Z+\eta(X)R(\xi,\widetilde{Y})Z\\
=&\fra{F+3}{4} R_1(\widetilde{X},\widetilde{Y})Z+\fra{F-1}{4} R_2(\widetilde{X},\widetilde{Y})Z+R_4(\widetilde{X},\widetilde{Y})Z+\fra{1}{2} R_5(\widetilde{X},\widetilde{Y})Z\\
&-\eta(Y) \{\ka(g(\widetilde{X},Z)\xi+\eta(Z)\widetilde{X})+\mu(g(h\widetilde{X},Z)\xi-\eta(Z)h\widetilde{X})\\
&\hspace{1.4cm}+\nu(g(\phi h Z,\widetilde{X})\xi-\eta(Z) \phi h \widetilde{X}) \}\\
&+\eta(X) \{\ka(g(\widetilde{Y},Z)\xi-\eta(Z)\widetilde{Y})+\mu(g(h\widetilde{Y},Z)\xi-\eta(Z)h\widetilde{Y})\\
&\hspace{1.4cm}+\nu(g(\phi h Z,\widetilde{Y})\xi-\eta(Z) \phi h \widetilde{Y})\}.
\end{align*}

After some calculations where we use the definition of the tensors $R_1,\ldots,R_6$ and that $\widetilde{X}=X-\eta(X)\xi$, it follows that the formula of the curvature tensor $R(X,Y)Z$ for any vector fields $X,Y,Z$ is:
\begin{align*}
R&(X,Y)Z=\frac{F+3}{4} R_1(X,Y)Z+\frac{F-1}{4} R_2(X,Y)Z+\left( \frac{F+3}{4}-\kappa \right) R_3(X,Y)Z\\
&+ R_4(X,Y)Z+ \frac{1}{2}R_5(X,Y)Z+(1-\mu) R_6(X,Y)Z\\
&- \nu  \{ \eta(X) \eta(Z) \phi h Y - \eta(Y) \eta(Z) \phi h X +g(\phi h X,Z) \eta(Y) \xi-g(\phi h Y,Z) \eta(X) \xi   \}.
\end{align*}

If we denote by  $R_8$ the factor that multiplies $\nu$, i.e., if we define the tensor
\begin{equation}\label{defR8}
\begin{split}
R_8(X,Y)Z  &= \eta(X) \eta(Z) \phi h Y - \eta(Y) \eta(Z) \phi h X \\
&+g(\phi h X,Z) \eta(Y) \xi-g(\phi h Y,Z) \eta(X) \xi,
\end{split}
\end{equation}
we can write the Riemann curvature tensor as
$$R=\frac{F+3}{4} R_1+\frac{F-1}{4} R_2+\left( \frac{F+3}{4}-\kappa \right) R_3+ R_4+ \frac{1}{2} R_5+(1-\mu) R_6-\nu R_8.$$
We know from Lemma 3.8 of \cite{C+MM2}  that $R_2=3(R_1+R_3)$, $R_5=0$ and $R_6=-R_4$ on every contact metric manifold, so we  obtain:
\begin{equation*}
R=F R_1+(F-\ka) R_3+\mu R_4-\nu R_8.
\end{equation*}
\end{remark}

This new way of writing the curvature tensor coincides with \eqref{eq-R-v2} thanks to the fact that $R_8=-R_7$ in every $3$-dimensional contact metric manifold $(M,\phi,\xi,\eta,g)$, which can be easily proved by checking that it is true for a $\phi$-basis $\{E,\phi E,\xi \}$.

\

Corollary \ref{prop-R-v2} also implies that the examples of
3-dimensional $(\ka,\mu,\nu)$-\emph{contact metric manifolds} with
non-constant functions given in \cite{kouf2008} have curvature
tensors written like (\ref{eq-R-v2}), so we only need to calculate
the $\phi$-sectional curvature $F$ in order to know the tensor
explicitly.

For instance, Example 4.2 of \cite{kouf2008}, which is a
3-dimensional $(\ka,\mu,\nu)$-\emph{contact metric manifold} with
$\ka=1-\fra{e^{2cx}z^2}{4}, \mu=2+e^{cx}z$ and $\nu=c\neq 0$ constant,
has Riemann curvature tensor $R=F R_1+(F-\ka) R_3+\mu R_4 +\nu
R_7,$ where $F$ is the $\phi$-sectional curvature
$$F=-\left(3+\fra{3}{2}c^2y^2+3 c^2yz+\fra{3cy}{2z}+\fra{3}{4z^2}+c^2z^2-\fra{1}{4}e^{2cx}z^2 \right).$$

\

\

We will now study how a $D_a$-homothetic deformation affects the
curvature tensor of a $(\ka,\mu,\nu)$-\emph{contact metric manifold}.

We already know from \cite{kouf2008} that  applying  a
$D_a$-homothetic deformation ($a>0$) to a  $(\ka,
\mu,\nu)$-\emph{contact metric manifold} yields a new  $(\overline{\ka},\overline{\mu}, \overline{\nu})$-\emph{contact metric manifold} with
\begin{equation} \label{kamunu-deformed}
\overline{\ka}=\fra{\ka+a^2-1}{a^2},\qquad
\overline{\mu}=\fra{\mu+2a-2}{a}, \qquad
\overline{\nu}=\fra{\nu}{a}.
\end{equation}

Applying  Corollary  \ref{eq-R-v2} we get that the deformed
manifold has a Riemann curvature tensor that can be written as:
$$\overline{R}=\overline{F} \overline{R}_1+(\overline{F}-\overline{\ka}) \overline{R}_3
+\overline{\mu} \overline{R}_4 +\overline{\nu} \overline{R}_7,$$
where $\overline{F}$ is the $\overline{\phi}$-sectional curvature
and $\overline{R}_1,\overline{R}_3,\overline{R}_4,\overline{R}_7$
are the already defined tensors on the deformed manifold.

If we use (\ref{kamunu-deformed}) and the fact that
$\overline{F}=\fra{1}{a} F-\fra{a-1}{a^2} (3a+1-\ka)$ under the
previous hypothesis, we conclude that the curvature tensor
$\overline{R}$ of the deformed $(\ka,\mu,\nu)$-\emph{contact metric
manifold} is
$$\overline{R}= \overline{f}_1 \overline{R}_1+ \overline{f}_3 \overline{R}_3
+ \overline{f}_4 \overline{R}_4+\overline{f}_7 \overline{R}_7,$$
where
\begin{align*}
\overline{f}_1&= \fra{1}{a} F-\fra{a-1}{a^2}(3a+1-\ka) ,\\
\overline{f}_3&= \fra{1}{a} F+\fra{1}{a^2}((a-2)\ka -4a^2+2a+2),\\
\overline{f}_4&=\fra{1}{a}(\mu+2a-2),\\
\overline{f}_7&=\fra{\nu}{a},
\end{align*} and $F$ is the $\phi$-sectional curvature of the
original manifold. Therefore, we can completely determine the curvature tensor of the deformed manifold, just by knowing the original $\kappa,\mu,\nu$ and $F$.

\section{Generalized $(\kappa,\mu,\nu)$-space
forms}\label{section-generalized}

We will extend the notion of \emph{generalized $(\kappa,\mu)$-space} form to englobe the $(\kappa,\mu,\nu)$-\emph{contact metric manifolds}.

\begin{definition}
A \emph{generalized} $(\kappa,\mu,\nu)$-\emph{space form} is an almost contact metric manifold whose curvature tensor can be written as
\begin{equation} \label{def-space.2}
R=f_{1}R_{1}+f_{2}R_{2}+f_{3}R_{3}+f_{4}R_{4}+f_{5}R_{5}+f_{6}R_{6}+f_7 R_7+f_8 R_8,
\end{equation} where $f_1,\ldots,f_8$ are arbitrary functions on $M$, $R_1,\ldots,R_6$ are the tensors in \eqref{def-space}, $R_7$ the one  that appears in \eqref{def-R7} and $R_8$ the one in \eqref{defR8}. We will denote it by $M(f_1,\ldots,f_8)$.
\end{definition}

Firstly, we study the contact metric case. We can easily obtain the following result:

\begin{proposition}
If $M(f_1,\ldots,f_8)$ is a  contact metric \emph{generalized $(\kappa,\mu,\nu)$-space form},
then it is a $(\kappa,\mu,\nu)$-\emph{contact metric manifold} with $\kappa=f_1-f_3$, $\mu=f_4-f_6$ and $\nu=f_7-f_8$.
\end{proposition}
\begin{proof}
We already knew from \cite{C+MM} that
$$(f_{1}R_{1}+f_{2}R_{2}+f_{3}R_{3}+f_{4}R_{4}+f_{5}R_{5}+f_{6}R_{6})(X,Y) \xi= $$
$$=(f_1-f_3) \{ \eta ( Y)
X-\eta ( X) Y\} +\mu \{ \eta ( Y)
hX-\eta ( X) hY\}.  $$ It is easy to check that
$$R_7(X,Y) \xi= -R_8(X,Y)\xi=\eta ( Y) \phi hX-\eta ( X)\phi hY.$$
Hence, we can conclude that the manifold is a $(f_1-f_3,f_4-f_6,f_7-f_8)$-contact metric manifold.
\end{proof}

Using Theorem 4.8 from \cite{C+MM} and Theorem 4.1 from \cite{kouf2008}, it is obvious that the following theorem holds true:

\begin{theorem}
Let $M(f_1,\ldots,f_8)$ be a contact metric \emph{generalized $(\kappa,\mu,\nu)$-space form} of dimension greater than or equal to $5$. Then
\begin{eqnarray}
f_1=\fra{f_6+1}{2}, &  f_2=\fra{f_6-1}{2}, & f_3=\fra{3f_6+1}{2},\nonumber\\
f_4=1  ,            &  f_5=\fra{1}{2},   & f_6 = \mbox{constant} >-1, \nonumber\\
\ka =&-f_6 &=\mbox{constant} < 1,\label{eqteodim5} \\
\mu=&1-f_6 &=  \mbox{constant} < 2,\nonumber\\
\nu=&f_7=f_8&=0, \nonumber\\
F=&2f_6-1& = \mbox{constant} >-3.\nonumber
\end{eqnarray}
Hence $M$ is a $(-f_6,1-f_6)$-space with constant $\phi$-sectional
curvature $F=2f_6-1>-1$.
\end{theorem}

Using  Lemma 3.8 from \cite{C+MM2}, we can easily see that the curvature tensor of a $3$-dimensional contact metric \emph{generalized
$(\ka,\mu,\nu)$-space form} $M(f_1,\ldots,f_8)$ can be written as
\begin{align*}
R&=f_{1}R_{1}+f_{2}R_{2}+f_{3}R_{3}+f_{4}R_{4}+f_{5}R_{5}+f_{6}R_{6}+f_7 R_7+f_8 R_8=\\
&=f_{1}R_{1}+3 f_{2}(R_1+R_3)+f_{3}R_{3}+f_{4}R_{4}-f_{6}R_{4}+f_7 R_7-f_8 R_7=\\
&=(f_1+3f_2)R_1+(f_3+3f_2)R_3+(f_4-f_6)R_6+(f_7-f_8)R_7.
\end{align*}

By an straightforward computation, we also have that the $\phi$-sectional curvature of that manifold would be $F=f_1+3f_2$, so its
curvature tensor could be written as
$$R=F R_1+(F-(f_1-f_3))R_3+(f_4-f_6)R_6+(f_7-f_8)R_7=F R_1+(F-\ka) R_3+\mu R_4 +\nu R_7,$$ which coincides with  equation (\ref{eq-R-v2}) from Corollary \ref{prop-R-v2}.

In conclusion, contact metric \emph{generalized $(\ka,\mu,\nu)$-space forms} are either $(\ka,\mu)$-spaces (in dimension greater than or equal to $5$)
or $(\ka,\mu,\nu)$-contact metric manifolds (in dimension $3$). This fact does not detract from the interest of defining such manifolds because there are \emph{generalized $(\ka,\mu,\nu)$-spaces} that
are not contact metric ones. For instance, G. Dileo and A. M. Pastore
study in \cite{dileo07} and \cite{dileo09} and A.M. Pastore and V. Salterelli in \cite{pastore10} almost Kenmotsu manifolds which are also
\emph{generalized $(\ka,\mu)$-spaces}  or \emph{generalized $(\ka,0,\nu)$-spaces}, though they use a different notation. They give examples in dimension $3$.

Almost cosymplectic $(\ka,\mu,\nu)$-spaces have also been widely studied. For $\mu=\nu=0$,  P. Dacko published   \cite{dacko2000}, where he proved that $\ka$ must be constant and H. Endo presented multiple results in \cite{endo94} and \cite{endo96}. This last author also examined in  \cite{endo97} and \cite{endo2002} these spaces for $\nu=0$ and  $\ka$,  $\mu$  constants. Afterwards, P. Dacko and Z. Olszak studied in \cite{dacko2005} and \cite{dacko2005-b}  almost cosymplectic $(\ka,\mu,\nu)$-spaces with $\ka,\mu$ and $\nu$ functions that only vary in the direction of the vector field $\xi$, presenting multiple examples.

Moreover, H. \"{O}zturk, N. Aktan and C. Murathan examine in \cite{OAM} the almost $\alpha$-cosymplectic $(\ka,\mu,\nu)$-spaces. They also provide an example of almost $\alpha$-cosymplectic $(\ka,\mu)$-space of dimension $3$ with non-constant functions $\ka$ and $\mu$.

Are these \emph{generalized $(\ka,\mu,\nu)$-spaces} also
\emph{generalized $(\ka,\mu,\nu)$-space forms}? It can be proved
that Theorem \ref{prop-R}, Proposition \ref{prop-F} and Corollary
\ref{prop-R-v2} are also true for $3$-dimensional
$(\ka,\mu,\nu)$-spaces with almost cosymplectic or almost Kenmotsu
structures. Therefore, the previously mentioned examples are
\emph{generalized $(\ka,\mu,\nu)$-space forms} with functions
$f_1=F$, $f_3=F-\ka$, $f_4=\mu$, $f_7=\nu$ and the rest zero.

\section{Conformally flat generalized $(\kappa,\mu,\nu)$-space
forms} \label{section-conformally}

In this section, we will give necessary and sufficient conditions
for a \emph{generalized $(\ka,\mu,\nu)$-space form} to be conformally
flat if its dimension is greater than or equal to $5$ and the tensor $h$
satisfies some properties.

It is easy to see that $R_4,\ldots,R_8$ must be zero if $h=0$. Therefore, a \emph{generalized
$(\ka,\mu,\nu)$-space form} with $h=0$ is a \emph{generalized Sasakian space form},
which was already studied under the hypothesis of conformal
flatness by U. K. Kim in \cite{unkyu}. One of the results he
proved was that a \emph{generalized Sasakian space form} $M(f_1,f_2,f_3)$
of dimension greater than or equal to $5$ is conformally flat if and only if
$f_2=0$. We can give a similar result in our case if $h \neq 0$.


We recall that a Riemann manifold is said to be conformally flat if it is locally conformal to a flat manifold. The \emph {Schouten tensor} of a manifold $M^{2n+1}$ is defined as
\begin{equation}\label{def-schouten}
L=-\fra{1}{2n-1} Q+\fra{\tau}{4n(2n-1)} I,
\end{equation} and the  \emph{Weyl tensor} as
\begin{equation}\label{def-weyl}
W(X,Y)Z=R(X,Y)Z-(g(LX,Z)Y-g(LY,Z)X+g(X,Z)LY-g(Y,Z)LX),
\end{equation} for all $X,Y,Z$ vector fields on $M$.

If the dimension of the manifold is greater than or equal to 5, it is well known that $M$ is conformally flat if and only if the Weyl tensor $W$ is identically zero. In dimension $3$, this tensor is always zero and the manifold is conformally flat if and only if the Schouten tensor is a Codazzi tensor, i.e., if it satisfies that  $(\nabla_X L)Y-(\nabla_Y L)X=0$, for all $X,Y$ vector fields on $M$. 

Before presenting the main theorem of this section, let us see a result which will be used in its proof:

\begin{lemma}
Let $M^{2n+1}(f_1,\ldots,f_8)$ be  a \emph{generalized $(\ka,\mu,\nu)$-space form}. If  $h \neq 0$, $h$ is symmetric
and $h \phi+\phi h=0$, then its Ricci operator is written as
\begin{equation} \label{eq-Q}
Q =(2n f_1+3 f_2-f_3) I-(3f_2+f_3(2n-1)) \eta \otimes \xi+
(f_4(2n-1)-f_6)+((2n-1) f_7-f_8) \phi h.
\end{equation}
Therefore, its scalar curvature is
\begin{equation}\label{eq-tau}
\tau=2n((2n+1)f_1+3f_2-2f_3).
\end{equation}
\end{lemma}

\begin{theorem}
Let $M^{2n+1}(f_1,\ldots,f_8)$ a \emph{generalized $(\ka,\mu,\nu)$-space
form} of dimension greater than or equal to $5$. If  $h \neq 0$, $h$ is symmetric
and $h \phi+\phi h=0$, then $M$ is conformally flat if and only if
$f_2=f_5 R_5=f_6=f_8=0$.
\end{theorem}
\begin{proof}
Substituting the formulas of the Ricci operator \eqref{eq-Q}  and the scalar curvature \eqref{eq-tau} on a \emph{generalized $(\ka,\mu,\nu)$-space
form} in the definition of the Schouten tensor \eqref{def-schouten} we get that
\begin{equation} \label{schouten}
\begin{split}
L=&-\fra{1}{2} \left(f_1+\fra{3}{2n-1}f_2 \right) I+\left(\fra{3}{2n-1}f_2 + f_3 \right) \eta \otimes \xi \\
&-\left(f_4-\fra{1}{2n-1}f_6 \right) h-\left(f_7-\fra{1}{2n-1} f_8 \right) \phi h.
\end{split}
\end{equation}
Using now equations \eqref{def-space.2} and \eqref{schouten} in the definition of the Weyl tensor \eqref{def-weyl}, we obtain  that it can be written as
\begin{equation} \label{weyl}
W=- \fra{3f_2}{2n-1} R_1+f_2 R_2-\fra{3f_2}{2n-1}
R_3+\fra{f_6}{2n-1}  R_4+f_5R_5+f_6 R_6+\fra{1}{2n-1} f_8 R_7+f_8 R_8.
\end{equation}

If $f_2=f_5 R_5=f_6=f_8=0$, it is obvious that $W=0$.

If $W=0$, then we have in particular that $W(X,\xi) \xi=0$ for
every vector field $X$ orthogonal to $\xi$. Thanks to equation
(\ref{weyl}), this means that
$$\fra{2(1-n)}{2n-1} (f_6 hX+f_8 \phi h X)=0.$$
Now, $2n+1 > 3$ so $f_6 hX+f_8 \phi h X=0$. Moreover, $h \neq 0$, so the vector fields $h X$ and $\phi h X$ are mutually orthogonal and not zero. Hence $f_6=f_8=0$ and (\ref{weyl}) could
be written as
\begin{equation} \label{weyl.2}
W=-\fra{3f_2}{2n-1} R_1+f_2 R_2 - \fra{3 f_2}{2n-1} R_3 + f_5 R_5.
\end{equation}

Taking now $X=\phi Y$ and $Z=Y$, with $Y$ an unit vector field
orthogonal to $\xi$, we obtain  $f_2=0$.

Therefore, the Weyl tensor would be $W=f_5 R_5=0$ and we conclude
the proof.
\end{proof}

\begin{remark}
The hypothesis $f_5 R_5=0$ of the previous theorem is not always equivalent to $f_5=0$ because the tensor $R_5$ could be identically zero, like it occurs in dimension $3$. However, if $f_5=0$ it is obvious that $f_5 R_5=0$ and if  $R_5=0$, then $f_5$ is an arbitrary function, so we can choose $f_5=0$ in particular.

The properties ``h is symmetric'' and  ``$h \phi +\phi h=0$'' are satisfied in some well-known cases. For example, if the manifold has a contact metric, an almost Kenmotsu or an almost cosymplectic structure.
\end{remark}

{\bf Acknowledgements.} The second author is partially supported by the MEC-FEDER grant MTM2007-61248 and the third one
 by the FPU scholarship program of the Ministerio de Educaci\'on, Spain.  Both of them are also supported by the PAI group FQM-327 (Junta de Andaluc\'{\i}a, Spain).

\end{document}